\documentclass[12pt,reqno]{amsart}
\usepackage[top=1in, bottom=1in, left=1in, right=1in]{geometry}
\usepackage{times, amsthm, amssymb, amsmath, amsfonts, amsthm, bm, graphicx, mathrsfs}
\usepackage{latexsym,wasysym}
\usepackage[usenames,dvipsnames]{color}
\usepackage{tikz}
\usepackage{graphicx}
\usepackage{float}
\usepackage{setspace}
\usepackage{enumerate}
\usepackage{bbm}
\usepackage{lineno}

\usepackage[T1,T5]{fontenc}

\newtheorem{theorem}{Theorem}[section]
\newtheorem{lemma}[theorem]{Lemma}
\newtheorem{proposition}[theorem]{Proposition}

\theoremstyle{definition}
\newtheorem{definition}[theorem]{Definition}
\newtheorem{example}[theorem]{Example}
 
\theoremstyle{remark}

\newcommand{\bbN}{\mathbbm{N}}

\newcommand{\bbR}{\mathbbm{R}}

\newcommand{\bbZ}{\mathbbm{Z}}

\newcommand{\bbk}{\mathbbm{k}}

\newcommand{\GL}{\operatorname{GL}}

\begin{document} 

\title{Counterexamples for Cohen--Macaulayness of Lattice Ideals}

\author{Laura Felicia Matusevich}
\address{Mathematics Department\\Texas A\&M University\\College Station, TX 77843}
\email{laura@math.tamu.edu}
\author{Aleksandra Sobieska}
\email{ola@math.tamu.edu}

\thanks{The authors were partially supported by NSF grant DMS 1500832.}


\begin{abstract}
Let $\mathscr{L}\subset \bbZ^n$ be a lattice, $I$ its corresponding
lattice ideal, and $J$ the toric ideal arising from the saturation of
$\mathscr{L}$. We produce infinitely many examples, in every
codimension, of pairs $I,J$ where one of these ideals is
Cohen--Macaulay but the other is not.
\end{abstract}

\parindent 0pt
\parskip .7em

\maketitle 

\section{Introduction} 

Let $n>2$ be an integer, and let $\bbk[x]= \bbk[x_1,\dots,x_n]$ be the polynomial
ring in $n$ commuting indeterminates over a field $\bbk$. 
In this article,  $\bbN = \{0,1,2, \dots \}$. 
%
We introduce our main object of study.

\begin{definition}
\label{def:latticeIdeal}
Let $\mathscr{L} \subset \bbZ^n$ be a lattice, that is, a subgroup of $\bbZ^n$.
The \emph{lattice ideal} corresponding to $\mathscr{L}$ is 
\begin{align*}
  I_{\mathscr{L}} := \langle x^a - x^b \, | \, a,b \in \bbN^n \text{ and }
  a-b \in \mathscr{L} \rangle \subset \bbk[x].
\end{align*}
\end{definition}

For $u \in \bbZ^n$, write $u=u_+-u_-$, where $u_+,u_-\in \bbN^n$ are
defined via 
$(u_+)_i =  \left\{ 
\begin{array}{ll}
u_i  & u_i \geq 0 \\ 
0 & u_i < 0
\end{array}
\right.$. 
It is not hard to check that 
\begin{align*}
  I_{\mathscr{L}} = \langle x^{u_+} - x^{u_-} \, | \, u \in \mathscr{L} \rangle.
\end{align*}

For $B \in \bbZ^{n\times m}$ of full rank
$m$, we set $\bbZ B$ to be the sublattice of $\bbZ^n$ generated over
$\bbZ$ by the columns of $B$. By construction, $\bbZ B$ is a rank $m$
lattice. Without loss of generality, we may assume that $B$ contains
no zero rows, for if a row of $B$ is zero, the lattice $\bbZ B$ can be
naturally embedded in $\bbZ^{n-1}$.

We assume throughout that we work with a \emph{positive} lattice
$\mathscr{L}$, meaning that $\mathscr{L} \cap \bbN^n = \{ 0 \}$.  
This positivity condition ensures that $I_{\mathscr{L}}$ is homogeneous
with respect to some $\bbZ$-grading of $\bbk[x]$ for which the values
$\deg(x_i)$, $i \in [n]:=\{1,\dots,n\}$, are positive integers.
Note that for a positive lattice $\bbZ B$, 
the matrix $B$ must be \emph{mixed}, meaning that every column contains a
strictly positive and a strictly negative entry. 
By~\cite[Corollary~2.2]{esBinomialIdeals}, the \emph{codimension} of
$I_{\mathscr{L}}$ equals the rank of $\mathscr{L}$.

The \emph{saturation} of a lattice $\mathscr{L}\subset \bbZ^n$ is 
  $\mathscr{L}^{\textrm{sat}}:= \{ u \in \bbZ^n \, | \, m \cdot u \in \mathscr{L} \text{ for some } m \in
  \bbN \}$.
We say that $\mathscr{L}$ is \emph{saturated} if $\mathscr{L} = \mathscr{L}^{\textrm{sat}}$. 
By~\cite[Corollary~2.2]{esBinomialIdeals}, if $\bbk$ is algebraically closed, then 
$I_{\mathscr{L}}$ is prime if and only if $\mathscr{L}$ is saturated. (Lattice
ideals corresponding to saturated lattices are always prime, it is the
converse of this assertion that requires the base field assumption.) Lattice ideals
corresponding to saturated lattices are known as~\emph{toric ideals}.

We remark that a lattice $\mathscr{L}$ is positive if and only if
$\mathscr{L}^{\textrm{sat}}$ is positive.

There is a strong relationship between the lattice ideals $I_{\mathscr{L}}$
and $I_{\mathscr{L}^{\textrm{sat}}}$, namely, $I_{\mathscr{L}^{\textrm{sat}}}$
is a minimal prime $I_{\mathscr{L}}$. Furthermore, if $\bbk$ is algebraically
closed, every associated prime of $I_{\mathscr{L}}$ is isomorphic to
$I_{\mathscr{L}^{\textrm{sat}}}$ by a rescaling of the
variables~\cite[Corollary~2.2]{esBinomialIdeals}. On the other hand,
toric ideals are better understood than lattice ideals. For instance,
there is a combinatorial/topological criterion to decide when a quotient by a toric ideal is
Cohen--Macaulay~\cite{thCohenMacaulay}, but there is no such
characterization of Cohen--Macaulayness for lattice ideals currently
available in the literature, beyond certain special
cases~\cite{pssyzygies,psgeneric,sscarf,eto}. 
At the root of most of these results is a topological method for
computing the graded Betti numbers of a lattice ideal as the ranks of
the homology groups of certain simplicial complexes (see
Lemma~\ref{pslemma}); however, these simplicial
complexes are not easily controlled in general.

In this article, we construct, for each codimension  $m\geq 2$,
infinitely many matrices $B$ for which one of $I_{\bbZ B}$, $I_{\bbZ
  B^{\textrm{sat}}}$ is Cohen--Macaulay (in fact, a complete intersection), but the
other one is not. This means that the most obvious place to look for a
Cohen--Macaulayness criterion for lattice ideals, namely the
associated toric ideals, does not directly yield positive results.


\section{Lattice Ideals in Codimension 2} 
\label{sec:Codim2}

In this section we study lattice ideals in codimension $2$. We recall
results from~\cite{pssyzygies} that characterize when such ideals are
(not) Cohen--Macaulay, and use them to construct examples.

Let $\mathscr{L} \subset \bbZ^n$ be a rank $m$ lattice, and let
$B=[b_{ij}]$ be an integer $n\times m$ matrix whose columns are a $\bbZ$-basis of $\mathscr{L}$.

Recall that an integer matrix is \emph{mixed} if every column contains a strictly
positive and a strictly negative entry. An integer matrix is \emph{dominating}
if it contains no square mixed submatrices. We emphasize that matrices
$B$ corresponding to $\bbZ$-generators of positive lattices are mixed.

\begin{theorem}{\cite[Theorem~2.9]{fsmixedbinoms}}
\label{thm:mixedDominating}
Let $\mathscr{L} \subset \bbZ^n$ be a rank $m$ lattice.  
The lattice ideal $I_{\mathscr{L}}$ is a complete intersection if and
only if
$\mathscr{L} = \bbZ B$ for some dominating matrix $B$.
In this case, $I_{\mathscr{L}}=\langle x^{u_+}-x^{u_-} \mid u \text{ is a
  column of } B\rangle$.
\end{theorem}

If $\mathscr{L} = \bbZ B$, use $b_i$ to denote the $i$th row of $B$. 
The collection $G_B:=\{b_1, \ldots, b_n \} \subseteq \bbZ^m$ is called
a \emph{Gale diagram} of $\mathscr{L}$.
Any $\bbZ$-basis for $\mathscr{L}$ yields a Gale diagram,
which means that Gale diagrams are unique up to the action of
$\GL_m(\bbZ)$. This elementary combinatorial object gives some insight
to the nature of $I_{\mathscr{L}}$: in the codimension $2$ case, Gale
diagrams can be used to restate Theorem~\ref{thm:mixedDominating} and
also to give a characterization for when a lattice ideal is Cohen--Macaulay.

We say that a Gale diagram  $G_B$ is \textit{imbalanced} if $b_{i1} = 0$ or $b_{i2} \geq 0$
for each $i = 1, \ldots, n$. 

\begin{theorem}{\cite[Lemma~3.1,~Proposition~4.1]{pssyzygies}}
\label{thm:codim2CM}
Let $\mathscr{L} \subset \bbZ^n$ be a rank $2$ lattice.  
\begin{enumerate}
\item
The lattice ideal $I_{\mathscr{L}}$ is a complete intersection if and only if
$\mathscr{L}$ has an imbalanced Gale diagram $G_B$. 
In this case, $I_{\mathscr{L}}=\langle x^{u_+}-x^{u_-} \mid u \text{ is a
  column of } B\rangle$.
\item
The lattice ideal $I_{\mathscr{L}}$ is \textbf{not} Cohen--Macaulay if
and only if it has a Gale diagram $G_B$ which intersects each of the four open quadrants of $\bbR^2$. 
\end{enumerate}
\end{theorem}

Because stretching or skewing a lattice $\bbZ B$ corresponds to
multiplying $B$ by a nonsingular  $2\times 2$ integer matrix, it is natural to
wonder how such an action transforms $G_{\bbZ B}$ and how this is
reflected in the corresponding lattice ideal. We illustrate below
how multiplication of $B$ by a nonsingular  $2\times 2$ integer matrix
can change a non-Cohen--Macaulay $I_{\bbZ B}$ to a complete intersection and vice-versa.

      \begin{figure}[!ht]
	\begin{center}
		\begin{tikzpicture}[scale=.5]
		
		\fill[blue, opacity=0.2] (-5/2,5) -- (5,5) -- (5,-5) -- (5/2,-5); 
		\fill[red, opacity=0.2] (5,10/3) -- (5,-5) -- (-5,-5) -- (-5,-10/3); 
		
		\node[label=right:$b_1$] (b1) at (3,1) {};
		\node[label=left:$b_2$] (b2) at (-2,4) {};
		\node[label=left:$b_3$] (b3) at (-2,-2) {};
		\node[label=right:$b_4$] (b4) at (3,-1.5) {};
		\node[label=left:\color{red}{$v$}] (v) at (2,-3) {};
		
		\draw[<->] (-5,0) -- (5,0); 
		\draw[<->] (0,-5) -- (0,5);
		
		\draw[thick, ->] (0,0) -- (b1);
		\draw[thick, ->] (0,0) -- (b2); 
		\draw[thick, ->] (0,0) -- (b3);
		\draw[thick, ->] (0,0) -- (b4);
		\draw[dashed, red, ->] (0,0) -- (v);
		\draw[dashed, red, <->] (5,10/3) -- (-5,-10/3);
		
		\end{tikzpicture}
	\end{center}
	\caption{A non-Cohen--Macaulay Gale diagram with new quadrants
          shaded}
\label{figure:prop:4x2nonCMtoCM}
\end{figure}
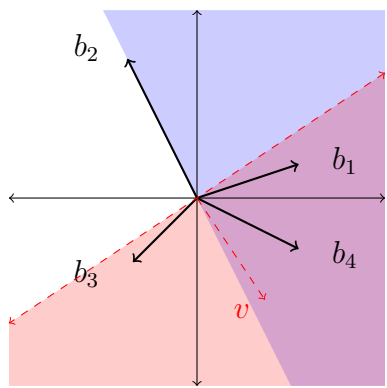

\begin{proposition} \label{prop:4x2nonCMtoCM}
For any $4 \times 2$ matrix $B$ such that $G_{B}$ touches all four open quadrants of
$\bbR^2$ (so that $I_{\bbZ B}$ is not Cohen--Macaulay), there exists
a nonsingular  $2\times 2$ integer matrix
$M$ such that $G_{BM}$ is imbalanced (so that 
$I_{\bbZ BM}$ is a complete intersection).  
\end{proposition}

\begin{proof}
Let $b_1,\dots,b_4$ be the rows of $B=[b_{ij}]$ and assume that $b_i$ lies in
the $i$th open quadrant of $\bbR^2$. 

If both sets $\{b_1,b_3\}$, $\{b_2,b_4\}$ are linearly dependent, let
$M = \left[\begin{smallmatrix} -b_{12} &  -b_{22}\\ \phantom{-}b_{11} &
    \phantom{-}b_{21} \end{smallmatrix} \right]$. Since $b_1$ and $b_2$ are
linearly independent, $\det(M) \neq 0$. But every row of $BM$ contains
a zero entry, so the corresponding Gale diagram $G_{BM}$ is imbalanced.

Now assume that $\{b_1,b_3\}$ are linearly independent. Then the cone
of nonnegative combinations of $b_1$ and $b_3$ contains either the
second or the fourth quadrant of $\bbR^2$. Suppose that it contains
the fourth quadrant. Then we can find $v=(v_1,v_2) \in \bbZ^2$ lying in the fourth
quadrant of $\bbR^2$ such that the angles between $v$ and $b_1$, and between $v$
and $b_3$ are both less than $\pi/2$, so that $b_1\cdot v \geq 0$ and
$b_3 \cdot v \geq 0$. Since they lie in the same quadrant, $b_4 \cdot
v \geq 0$. Now let $M = \left[ \begin{smallmatrix}  \phantom{-}b_{22} & v_1 \\
    -b_{21} & v_2 \end{smallmatrix} \right]$. Since $v$ lies in the
fourth quadrant, and $(b_{22}, -b_{21})$ lies in the first, they are
linearly independent, and so $\det(M)\neq 0$. By construction,
$G_{BM}$ is imbalanced. See Figure~\ref{figure:prop:4x2nonCMtoCM} for
a pictorial illustration of this argument.
\end{proof}

In the next proposition, we determine which imbalanced Gale diagrams
can be transformed into Gale diagrams of non-Cohen--Macaulay lattice
ideals. If $b \in \mathbb{R}^2 \smallsetminus \{ 0 \}$, the \emph{ray}
spanned by $b$ is defined to be $\{ t b
\mid t \in \mathbb{R}, t\geq 0\}$. If $B \in \mathbb{Z}^{n\times 2}$, we
consider the collection of rays spanned by the rows of $B$, and
associate this collection to the Gale diagram $G_B$. Since we
assume all rows of $B$ are nonzero, none of these rays is a point.

\begin{proposition} \label{CItononCMcodim2}
Let $B$ be an $n \times 2$-matrix such that $G_{B}$ is imbalanced. 
There exists a nonsingular  $2\times 2$ integer matrix
$M$ such that $G_{BM}$ meets the four
open quadrants of $\mathbb{R}^2$ (which implies that $I_{\mathbb{Z}
  BM}$ is not Cohen--Macaulay) if and only if $G_B$ spans more than three rays.
\end{proposition}

\begin{proof}
Note that if $M$ is a nonsingular  $2\times 2$ integer matrix,
then $G_B$ and $G_{BM}$ span the
same number of rays. If $G_{BM}$ meets the four open quadrants
of $\mathbb{R}^2$, it spans at least four rays, and therefore, so does $G_B$.

Now assume that  $G_B$ spans more than three rays.
We first consider the case where $G_B$ is contained in the coordinate
axes. Then $G_B$ must span all four half axes. Using $M =\left[ \begin{smallmatrix} 1 & -1 \\ 1 &
  \phantom{-}1 \end{smallmatrix} \right]$, we see that $G_{BM}$ meets the four
open quadrants of $\mathbb{R}^2$.

Now consider the case where our imbalanced $G_B$ is not contained in the coordinate
axes (and spans at least four rays). Then, 
since $\bbZ B \cap \bbN^n = 0$, we have that both
the first and second open quadrants of $\bbR^2$ 
contain elements of $G_B$, and there is some $b_j$ where $b_{j1} = 0$ and $b_{j2} <
0$. 

Consider the rightmost and leftmost elements of
$G_{B}$, that is, the vectors with the smallest positive and largest
negative slopes. Denote them $b_a = (b_{a1},b_{a2})$ and $b_d =
(b_{d1},b_{d2})$ respectively. Alternatively, we can characterize this
by  
  $$\arccos (b_{a1}/|b_a|) < \arccos(b_{i1}/|b_i|) < \arccos(b_{d1}/|b_d|)$$
for all $i$ such that $b_i$ lies in the upper half plane. 

Since $G_B$ spans at least four rays, there exists $b_c=(b_{c1},b_{c2})
\in G_B$ between $b_a$ and $b_d$, where we formalize ``between" to mean that 
  $\arccos (b_{a1}/|b_a|) < \arccos(b_{c1}/|b_c|) < \arccos(b_{d1}/|b_d|)$.
As this is moderately unappealing, the reader may choose to simply
visualize starting at the positive $x$-axis and sweeping
counter-clockwise, and declaring a vector between two others if they
encounter it after the first vector but before the second.  

Now choose $s = (s_1,s_2) \in \bbZ^2$ between $b_a$ and $b_c$ and $t =
(t_1,t_2) \in \bbZ^2$
between $b_c$ and $b_d$. We may assume that $s_1 > 0$ and $t_1 <
0$. Note that $s_2, t_2 > 0$.  

The matrix
  \[
  M = \begin{bmatrix} 
  -s_2 & \phantom{-}t_2 \\ 
  \phantom{-}s_1 & -t_1
  \end{bmatrix}
  \]
is nonsingular, since $s$ and $t$ are linearly independent (they belong
to adjacent quadrants of $\mathbb{R}^2$).

As stated previously, $G_B$ contains at least one element whose
first entry is zero, and whose second entry is negative.
Our construction yields the following sign pattern:
  \[
  \begin{bmatrix} 
  b_{a1} & b_{a2} \\
  b_{c1} & b_{c2} \\
  b_{d1} & b_{d2} \\ 
  0 & b_{j2}
  \end{bmatrix}  
  M
  = 
  \begin{bmatrix} 
  - & + \\
  + & + \\
  + & - \\ 
  - & - 
  \end{bmatrix}.
  \]
We conclude that the Gale diagram $G_{BM}$ has an arrangement of vectors in
all four open quadrants of $\mathbb{R}^2$, so $I_{\bbZ BM}$ is not
Cohen--Macaulay. Figure~\ref{figure:CItononCMcodim2} illustrates this argument.
\end{proof}

      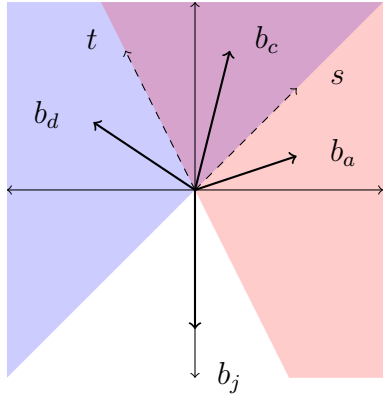
\begin{figure}[!ht]
      \begin{center}
      \begin{tikzpicture}[scale=.5]

      \fill[blue, opacity=0.2] (-5,-5) -- (5,5) -- (-5,5) -- (-5,-5);
      \fill[red, opacity=0.2] (-2.5,5) -- (5,5) -- (5,-5) -- (2.5,-5) -- (-2.5,5);

      \node (ori) at (0,0) {};
      \node[label=left: $b_d$] (bd) at (-3,2) {};
      \node[label=right:$b_c$] (bc) at (1,4) {};
      \node[label=right:$b_a$] (ba) at (3,1) {};
      \node[label=below right:$b_j$] (bj) at (0,-4) {};
      \node[label=left:$t$] (t) at (-2,4) {};
      \node[label=right:$s$] (s) at (3,3) {};

      \draw[<->] (-5,0) -- (5,0); 
      \draw[<->] (0,-5) -- (0,5);
      \draw[thick, ->] (0,0) to (ba);
      \draw[thick, ->] (0,0) to (bc);
      \draw[thick, ->] (0,0) to (bd);
      \draw[thick, ->] (0,0) to (bj);
      \draw[dashed, ->] (0,0) to (s);
      \draw[dashed, ->] (0,0) to (t);

      \end{tikzpicture}
      \end{center}
      \caption{An imbalanced Gale diagram with new quadrants in red
        and blue}
      \label{figure:CItononCMcodim2}
      \end{figure}

The main result in this section brings together
Propositions~\ref{prop:4x2nonCMtoCM} and~\ref{CItononCMcodim2}.

\begin{theorem}
\label{thm:codim2case}
There are infinitely many examples of rank $2$ non-saturated lattices
$\mathscr{L}$ such that one of $I_{\mathscr{L}}$,
$I_{\mathscr{L}^{\textrm{sat}}}$ is a complete intersection, and the
other one is not Cohen--Macaulay.
\end{theorem}

\begin{proof}
Note that there are infinitely many $4\times 2$ integer matrices $B$
with Gale diagram meeting the four open quadrants of $\mathbb{R}^2$
and whose columns span a saturated lattice. Let $M$ as in
Proposition~\ref{prop:4x2nonCMtoCM}. Define
$\mathscr{L} \subset \bbZ^4$ to be the lattice spanned by the columns
$BM$, so that $\mathscr{L}^{\textrm{sat}}$ 
is the lattice spanned by the columns of $B$. We see
that $I_{\mathscr{L}}$ is a complete intersection, while
$I_{\mathscr{L}^{\textrm{sat}}}$ is not Cohen--Macaulay.

Similarly, there are infinitely many  
$B \in \mathbb{Z}^{n\times 2}$ whose columns span a saturated
lattice, and whose Gale diagram is imbalanced and spans at least four
rays. With $M$ as in Proposition~\ref{CItononCMcodim2}, and
$\mathscr{L} \subset \bbZ^n$ the lattice spanned by the columns of $BM$, 
we have that
$I_{\mathscr{L}}$ is not
Cohen--Macaulay, and $I_{\mathscr{L}^{\textrm{sat}}}$ is a complete intersection.
\end{proof}

\section{Lattice Ideals in Codimension $\geq$ 3} 

The goal of this section is to generalize Theorem~\ref{thm:codim2case}
to higher ranks. We begin by reviewing relevant results on Betti
numbers of lattice ideals. 

Let $B \in \mathbb{Z}^{n\times m}$ of full rank $m$, and 
consider $R = \bbk[x]/I_{\bbZ B}$, $\Gamma = \bbZ^n/\bbZ B$. The
congruence classes of $\mathbb{Z}^n$ modulo $\Gamma$ are called
\emph{fibers}. The polynomial ring $\bbk[x]$ is $\Gamma$-graded by
setting $\deg(x^u)$ to be the fiber of $\Gamma$ containing $u$. The
ideal $I_{\mathbb{Z}B}$ is homogeneous with respect to this grading,
  and moreover the quotient $R$ is \emph{finely graded}, meaning that its
  graded pieces have dimension at most one. If $C$ is a fiber of
  $\Gamma$, we denote the corresponding graded piece of $R$ by $R_C$.
This grading gives a decomposition
$\textrm{Tor}_j^{\bbk[x]}(R,\bbk) = \displaystyle \bigoplus_{C \in \Gamma}
\textrm{Tor}_j^{\bbk[x]}(R,\bbk)_C$. The \emph{multi-graded Betti
  number} of $R$  in degree $C$ is  
$\beta_{j,C} = \beta_{j,C}(R)= \dim_\bbk \textrm{Tor}_j^{\bbk[x]}(R,\bbk)_C$. 

The Betti number $\beta_{j,C}$ can be computed using simplicial homology.
For each fiber $C$ of $\Gamma$, define a simplicial complex  $\Delta_C$ on
$[n]$, where $F \subseteq [n]$ is a face of $\Delta_C$ if and
only if $C$ contains a nonnegative vector $a$ whose support contains $F$.  

\begin{lemma}{\cite[Lemma~4.1]{ahkoszul}, \cite[Lemma~2.1]{pssyzygies}}\label{pslemma}
The multigraded Betti number $\beta_{j+1,C}(\bbk[x]/I_{\bbZ B})$ equals the rank of the
$j$-th reduced homology group $\tilde{H}_j(\Delta_C;\bbk)$ of the
simplicial complex $\Delta_C$.  
\end{lemma} 

We recall that $R$ is Cohen--Macaulay if and only if $\beta_{j,C}(R)=0$
whenever $j>\textrm{codim}_{\bbk[x]}(R)$.

Since the simplicial complex $\Delta_C$ is defined using the nonnegative
elements of the fiber $C$, we consider only those elements when
working with specific fibers.

\begin{example}
Consider the lattice $\bbZ B \subset \bbZ^4$ where 
$B = \scriptsize{\begin{bmatrix} \phantom{-}2 & 3 & -1 & -4 \\ -1 & 3
    & \phantom{-}5 & -7 \end{bmatrix}^T}.$ 
For the fiber 
$C \supseteq \{~(2,~6,~5,~0), (3,~3,~0,~7), (0,~3,~6,~4), (1,~0,~1,~11)~\} $, 
the simplicial complex $\Delta_C$ is a hollow tetrahedron. By
Lemma~\ref{pslemma}, we see that  
$\beta_{3,C}(\bbk[x]/I_{\bbZ B}) = \textrm{rank}_\bbk~\tilde{H}_2(\Delta_C;\bbk) = 1$. 
Since $R$ has codimension $2$, we conclude that $R$ is not Cohen--Macaulay.   
\end{example}

In the previous example, the high syzygy constructed occurs in the
fiber $C$ containing $(B_1)_+ + (B_2)_+$, where 
$B_1$ and $B_2$ are the columns of $B$. This turns out to be the case in general.  

\begin{proposition} 
\label{prop:highSyzygyDegree}
Let $B$ be an $n \times 2$ integer matrix of rank $2$ with columns
$B_1$ and $B_2$ 
such that $G_{B}$ meets the four open quadrants of $\mathbb{R}^2$. 
Let $C_\circ$ be the fiber of $\bbZ^n/\bbZ B$ containing $(B_1)_+ + (B_2)_+$.  
Then $\beta_{3,C_\circ}(\bbk[x]/I_{\bbZ B}) > 0$.
\end{proposition}

\begin{proof}
Since $G_B$ meets the four open quadrants of $\bbR^2$, we may assume
that for $1 \leq i \leq 4$, the $i$th row of $B$ meets the $i$th
quadrant of $\bbR^2$. For convenience, we write the submatrix of $B$
consisting of its first four rows as:
\[
\begin{bmatrix}
t_1 & t_2 \\
y_1 & y_2 \\
z_1 & z_2 \\
w_1 & w_2
\end{bmatrix} \textrm{ with sign pattern } 
\begin{bmatrix}
+ & + \\
+ & - \\
- & + \\ 
- & - \\
\end{bmatrix}.
\]

An element of $C_\circ$, different from $(B_1)_++(B_2)_+$, 
whose entries are all nonnegative can be written as $B
\scriptsize{\begin{bmatrix} u \\ 
    v \end{bmatrix} } + (B_1)_+ + (B_2)_+$, where 
$0 \neq \scriptsize{\begin{bmatrix} u \\ v \end{bmatrix} }\in \bbZ^2$ is such
that  
\begin{align}
B \begin{bmatrix} u \\ v \end{bmatrix} + 
(B_1)_+ + (B_2)_+ & \geq 0  \textrm{ coordinatewise.} \label{fullsupp}
\end{align}

In particular, restricting to the first four rows of $B$, we have that

\begin{align}
(u+1)t_1 + (v+1)t_2 & \geq 0 \label{fullsupp1}\\
(u+1)y_1 + vy_2 & \geq 0 \label{fullsupp2}\\
uz_1 + (v+1)z_2 & \geq 0 \label{fullsupp3}\\
uw_1 + vw_2 & \geq 0 \label{fullsupp4}
\end{align}

Because both $w_1, w_2 < 0$ and either $u$ or $v$ is nonzero,
equation~\eqref{fullsupp4} implies that at least one of $u,v < 0$.  
Suppose then that $u < 0$. Since $y_1 > 0$, we have $(u+1)y_1 \leq
0$, so $(u+1)y_1 + v y_2 \leq v y_2$. By~\eqref{fullsupp2}, we have $v y_2 \geq 0$. As
$y_2 < 0$, we see that $v \leq 0$. 
If $v=0$, equation~\eqref{fullsupp2} reduces to $(u+1)y_1 \geq 0$, so
$(u+1)y_1 = 0$, which means that $u=-1$. Otherwise, $v<0$, but then both
$(u+1)t_1 , (v + 1)t_2 \leq 0$, so by~\eqref{fullsupp1}, $(u+1)t_1  = 
(v + 1)t_2 = 0$, and therefore $u=v=-1$. 

The case where $v < 0$ is similar, leading to the possibilities $u=0,
v=-1$, and $u=v=-1$.

We conclude that the only coordinatewise nonnegative elements of $C_\circ$
other than $(B_1)_+ + (B_2)_+$ are $(B_1)_+ + (B_2)_-$, $(B_1)_- +
(B_2)_+$, and $(B_1)_- + (B_2)_-$, obtained by
taking $\scriptsize{ \begin{bmatrix} u \\ v \end{bmatrix} }$ equal to
$\scriptsize{ \begin{bmatrix} \phantom{-}0 \\ -1 \end{bmatrix}}$,
$\scriptsize{\begin{bmatrix} -1 \\ \phantom{-}0 \end{bmatrix}}$, and
$\scriptsize{ \begin{bmatrix} -1 \\ -1 \end{bmatrix} }$,
respectively. Consequently, the maximal faces of $\Delta_{C_\circ}$ are 
\begin{align*}
\{ i \mid b_{i1} > 0 \textrm{ or } b_{i2} > 0\} &\supseteq \{1,2,3\}, \qquad
\{ i \mid b_{i1} > 0 \textrm{ or } b_{i2} < 0\} \supseteq \{1,2,4\}, \\
\{ i \mid b_{i1} < 0 \textrm{ or } b_{i2} > 0\} &\supseteq \{1,3,4\}, \qquad
\{ i \mid b_{i1} < 0 \textrm{ or } b_{i2} < 0\} \supseteq \{2,3,4\}.
\end{align*}
We see that $\Delta_{C_\circ}$ has a hollow tetrahedron as a deformation
retract, and hence $\beta_{3,C_\circ}(\bbk[x]/I_{\bbZ B}) =1 $.
\end{proof}

We wish to construct examples of lattice ideals in codimension greater
than $2$ generalizing Theorem~\ref{thm:codim2case}. We do this using
block matrices.

Suppose $G$ is an $n \times m$ matrix of full rank of the form  
\[
G = \left[
\begin{array}{c|c|c|c}
G_1 & 0 & 0 & 0 \\
\hline
0 & G_2 & 0 & 0 \\ 
\hline
0 & 0 & \ddots & 0 \\ 
\hline
0 & 0 & 0 & G_r
\end{array}
\right],
\]

where each $G_i$ is an $n_i \times m_i$-matrix and the columns of $G$
form a basis for the lattice $\bbZ G$.  

Because of the block structure of $G$, $\bbZ G = \bbZ G_1 \oplus \bbZ
G_2 \oplus \cdots \oplus \bbZ G_r$. Any fiber $C$ of $\bbZ^n / \bbZ G$
has the form $C = C_1 \times \cdots \times C_r$, where $C_i$ is a
fiber of
$\bbZ^{n_i}/\bbZ G_i$. In particular, $C$ has coordinatewise positive (resp. nonnegative) elements of the form $(\alpha_1, \ldots, \alpha_r)$, where
$\alpha_i$ is a coordinatewise positive (resp. nonnegative) element of the fiber $C_i
\in \bbZ^{n_i}/ \bbZ G_i$. The resulting complex of supports
$\Delta_C$ is then $\Delta_{C_1} \star \Delta_{C_2} \star \cdots \star
\Delta_{C_r}$, where $\star$ denotes the join of two simplicial
complexes. We recall that the join of two simplicial complexes $\Delta$
and $\Delta'$ is the simplicial complex on the union of their
vertices that has faces $F \cup F'$ where $F$ is a face of $\Delta$
and $F'$ is a face of $\Delta'$.  If $\Delta'$ is the zero-dimensional
simplicial complex with two vertices, then $\Delta \star \Delta'$ is
the \emph{suspension} of $\Delta$, denoted $\Sigma \Delta$. In terms of
reduced homology, it is known that
\begin{equation}
\label{eqn:homologySuspension}
\tilde{H}_{j+1} (\Sigma \Delta; \bbk) \cong  \tilde{H}_j(\Delta;\bbk)
\qquad \textrm{for } j\geq-1.
\end{equation}

Using block matrices, it is easy to construct non-Cohen--Macaulay
lattice ideals of any codimension $m\geq 2$.

\begin{proposition}
\label{prop:nonCMSuspension}
Let $B_\circ \in
\bbZ^{n\times 2}$ such that $G_{B_\circ}$
meets the four open quadrants of $\bbR^2$ and let
and $H = \scriptsize{\begin{bmatrix} \phantom{-}1 \\
    -1 \end{bmatrix} } \in \bbZ^{2\times 1}$.
For $m>2$, let
\[
B = \left[
\begin{array}{c|c|c|c}
B_\circ & 0 & 0 & 0 \\
\hline
0 & H & 0 & 0 \\ 
\hline
0 & 0 & \ddots & 0 \\ 
\hline
0 & 0 & 0 & H
\end{array}
\right] \in \bbZ^{(n+2(m-2))\times m}.
\]
Note that $\bbZ B$ is a positive lattice if and only if $\bbZ B_\circ$
is a positive lattice.
Set $\Gamma = \bbZ^{n+2(m-2)}/ \bbZ B$. If $C$ is the fiber of
$\Gamma$ containing $(B_1)_++(B_2)_++\cdots+(B_m)_+$, then
\[
\beta_{m+1,C}(\bbk[x_1,\dots,x_{n+2(m-2)}]/I_{\bbZ B}) =
1.
\] 
Consequently $I_{\bbZ B}$ is not Cohen--Macaulay.
\end{proposition}

\begin{proof}
Let $C_\circ$ be the fiber in $\bbZ^n/\bbZ B_\circ$ containing
$((B_\circ)_1)_+ + ((B_\circ)_2)_+$, and $\Delta_\circ =
\Delta_{C_\circ}$. By (the proof of)
Proposition~\ref{prop:highSyzygyDegree},
$\tilde{H}_3(\Delta_\circ;\bbk)$ has rank $1$.

Let $C'$ be the fiber in $\bbZ^2/ \bbZ H$ containing $(1,0)$. We see that $C'=\{(1,0), (0,1)\}$,
and therefore $\Delta':=\Delta_{C'}$ is the zero dimensional
simplicial complex with two vertices. 

Then $\Delta_C = \Delta_\circ \star \Delta' \star \cdots \star
\Delta'$ is a sequence of suspensions of
$\Delta_\circ$. Repeatedly applying~\eqref{eqn:homologySuspension},
and using that $\tilde{H}_3(\Delta_\circ;\bbk)$ has rank $1$, we see
that $\tilde{H}_{m+1}(\Delta_C;\bbk)$ has rank $1$, and therefore 
\[
\beta_{m+1,C}(\bbk[x_1,\dots,x_{n+2(m-2)}]/I_{\bbZ B}) = 1.
\]
\end{proof}

The property of being a complete intersection also behaves well
with respect to block matrices for lattice ideals. This is stated precisely as follows.

\begin{lemma}
\label{lemma:blockCI}
Let $B^{(1)} \in \bbZ^{n_1\times m_1}$ and $B^{(2)} \in \bbZ^{n_2\times m_2}$
have rank $m_1$ and $m_2$ respectively. Let 
$B = \left[ \scriptsize{\begin{array}{c|c}
B^{(1)} & 0 \\
\hline
0 & B^{(2)} 
\end{array}} \right] \in \bbZ^{(n_1+n_2)\times(m_1+m_2)}
$. 
The binomials corresponding to the columns of $B$ generate $I_{\bbZ
  B}$ if and only if the binomials corresponding to the columns of
$B^{(i)}$ generate $I_{\bbZ B^{(i)}}$ for $i=1,2$. 
\end{lemma}

\begin{proof}
For ease in the notation, we consider 
\[
I_{\bbZ B^{(1)}} \subset \bbk[x_1,\dots,x_{n_1}] = \bbk[x] \textrm{ and } 
I_{\bbZ B^{(2)}} \subset \bbk[y_1,\dots,y_{n_2}]=\bbk[y] 
\]
so that $I_{\bbZ B} \subset
\bbk[x_1,\dots,x_{n_1},y_1,\dots,y_{n_2}]=k[x,y]$. We also consider
$\bbk[x], \bbk[y] \subset \bbk[x,y]$.

The lattice ideal $I_{\bbZ B}$ is generated by binomials
$x^{u_+}y^{v_+}-x^{u_-}y^{v_-}$ for $u \in \bbZ B^{(1)} \subset
\bbZ^{n_1}$ and $v \in \bbZ B^{(2)} \subset \bbZ^{n_2}$. Note that
$x^{u_+}y^{v_+}-x^{u_-}y^{v_-} = x^{u_+}(y^{v_+}-y^{v_-}) +
y^{v_-}(x^{u_+}-x^{u_-})$. This implies that $I_{\bbZ B} = \bbk[x,y]
\cdot I_{\bbZ B^{(1)}}+ \bbk[x,y] \cdot I_{\bbZ B^{(2)}}$. Moreover, 
$I_{\bbZ B^{(1)}} = I_{\bbZ B} \cap \bbk[x]$ and $I_{\bbZ B^{(2)}} =
I_{\bbZ B} \cap \bbk[y]$. Our statement follows from these relationships.
\end{proof}

We now come to the main result in this article.

\begin{theorem}
\label{thm:generalCase}
Let $m \geq 2$ be an integer. 
There are infinitely many examples of rank $m$ non-saturated lattices
$\mathscr{L}$ such that one of $I_{\mathscr{L}}$,
$I_{\mathscr{L}^{\textrm{sat}}}$ is a complete intersection, and the
other one is not Cohen--Macaulay.
\end{theorem}

\begin{proof}
When $m=2$, this is Theorem~\ref{thm:codim2case}. Let $B_\circ \in
\bbZ^{n\times 2}$ and $H = \scriptsize{\begin{bmatrix} \phantom{-}1 \\
    -1 \end{bmatrix} } \in \bbZ^{2\times 1}$. Suppose now $m>2$ and consider the matrix 
\[
B = \left[
\begin{array}{c|c|c|c}
B_\circ & 0 & 0 & 0 \\
\hline
0 & H & 0 & 0 \\ 
\hline
0 & 0 & \ddots & 0 \\ 
\hline
0 & 0 & 0 & H
\end{array}
\right] \in \bbZ^{(n+2(m-2))\times m},
\]

Let $B_\circ \in \bbZ^{4\times 2}$ whose Gale diagram meets the four
open quadrants of $\bbR^2$, and whose columns span a saturated lattice
(there are infinitely many such matrices). Let $M_\circ$ be as in
Proposition~\ref{prop:4x2nonCMtoCM}, and consider the nonsingular matrix
\begin{equation}
\label{eqn:blockM}
M = \left[
\begin{array}{c|c|c|c}
M_\circ & 0 & 0 & 0 \\
\hline
0 & 1 & 0 &  0 \\
\hline
0 & 0 & \ddots & 0 \\
\hline 
0 & 0 & 0 & 1   
\end{array}
\right] \in \bbZ^{m\times m}.
\end{equation}
Then we have
\[
BM =  \left[
\begin{array}{c|c|c|c}
B_\circ M_\circ & 0 & 0 & 0 \\
\hline
0 & H & 0 &  0 \\
\hline
0 & 0 & \ddots & 0 \\
\hline 
0 & 0 & 0 & H   
\end{array}
\right]\in \bbZ^{(n+2(m-2))\times m}.
\]
If $\mathscr{L}$ is the lattice spanned by the columns of $BM$, then
$\mathscr{L}^{\textrm{sat}}$ is the lattice spanned by the columns of
$B$. By Lemma~\ref{lemma:blockCI}, $I_{\mathscr{L}}$ is a complete
intersection, and by Proposition~\ref{prop:nonCMSuspension},
$I_{\mathscr{L}^{\textrm{sat}}}$ is not Cohen--Macaulay.

Now let $B_\circ \in \bbZ^{n \times 2}$ whose columns span a saturated
  lattice, and whose Gale diagram is imbalanced and spans at least
  four rays (there are infinitely many such matrices). Let $M_\circ$
  be as in Proposition~\ref{CItononCMcodim2}, and construct $M$
  using~\eqref{eqn:blockM}. 
If $\mathscr{L}$ is the lattice spanned by the columns of $BM$, then
$\mathscr{L}^{\textrm{sat}}$ is the lattice spanned by the columns of
$B$. By Proposition~\ref{prop:nonCMSuspension} $I_{\mathscr{L}}$ is
not Cohen--Macaulay, and by 
Lemma~\ref{lemma:blockCI} $I_{\mathscr{L}^{\textrm{sat}}}$ is a
complete intersection.
\end{proof}

\bibliographystyle{plain}
\bibliography{biblio.bib} 

\end{document}